\documentclass[twoside,11pt,reqno]{amsart}

\usepackage{macros}
\usepackage{color,soul}

\begin{document}

\title{On The Finite Generation of Relative Cohomology for Lie Superalgebras}
\author{Andrew Maurer}
\address{Department of Mathematics\\University of Georgia\\Athens, GA 30602}
\email{andrew.b.maurer@gmail.com}
\date{\today}
\subjclass[2010]{17B10}
\keywords{Lie superalgebra, finite generation, cohomology, support variety, spectral sequence.}

\begin{abstract}

  The author establishes finite-generation of the cohomology ring of a classical Lie superalgebra relative to an even subsuperalgebra. A spectral sequence is constructed to provide conditions for when this relative cohomology ring is Cohen-Macaulay. With finite generation established, support varieties for modules are defined via the relative cohomology, which generalize those of \cite{BKN-1}.
\end{abstract}

\maketitle

\section{Introduction}
\label{sec:intro}

\subsection{}Establishing finite generation of cohomology rings is a powerful result in representation theory which links cohomology theory with commutative algebra and algebraic geometry. For example, Evens \cite{Evens-cohomology-ring} and Venkov \cite{MR0108788} each independently proved that the cohomology ring of a finite group is finitely generated. This result was used by Quillen \cite{MR0298694}, Carlson \cite{MR723070}, Chouinard \cite{MR0401943}, and Alperin-Evens \cite{MR621284} to study the cohomology variety of the finite group. This allowed those listed, among others, to use techniques from classical algebraic geometry in the study of representation theory of finite groups. Similar work has been carried out in other contexts; by Friedlander-Parshall \cite{FP-unipotent,MR860682} for restricted Lie algebras, and by Friedlander-Suslin \cite{MR1427618} for finite-dimensional cocommutative Hopf algebras.

Relative cohomology, as defined by Hochschild \cite{MR0080654} is less understood than ordinary cohomology. For instance, the cohomology ring of a finite group relative to a subgroup need not be finitely generated. Indeed, Brown \cite{MR1310744} provided an example of a finite group whose relative cohomology is infinitely generated. Surprisingly, in the case $\g = \even{\g} \oplus \odd{\g}$ is a (finite-dimensional) classical Lie superalgebra, the cohomology ring of $\g$ relative to $\even{\g}$ is always finitely generated. Specifically, Boe-Kujawa-Nakano \cite{BKN-1} realized this relative cohomology ring as the invariants of a polynomial ring under the action of a reductive group. In fact, in the case of Lie superalgebras, ordinary cohomology is often times finite-dimensional as a vector space, as proved by Fuks-Leites \cite{fuks1984cohomology}
. This implies relative cohomology rings carry more representation theoretic information than their ordinary counterparts. Furthermore, Boe-Kujawa-Nakano \cite{BKN-1} demonstrated the atypicality of a supermodule -- a combinatorial invariant defined by Kac-Wakimoto \cite{MR1327543} -- is realized as the dimension of the support variety of that module. The geometrization of combinatorial ideas makes support variety theory useful and powerful.

One of the main results of this paper asserts that for a classical Lie superalgebra, cohomology rings relative to even subalgebras are finitely-generated over $\CC$, and the relative cohomology of a finite-dimensional module is a Noetherian module for this ring. In proving the main theorem, a spectral sequence is constructed  which relates relative Lie algebra cohomology to odd degree elements of the Lie superalgebra in an interesting way. The main theorem paves the way to define and investigate support varieties for supermodules relative to a broader class of subalgebras. The importance of this result is apparent in that cohomology relative to an even subalgebra provides a middle ground between the case of absolute cohomology of Fuks-Leites and cohomology relative to $\even{\g}$ of \cite{BKN-1}.

\subsection{Overview of paper}
\label{sec:overview}

Let $\g = \even{\g} \oplus \odd{\g}$ be a Lie superalgebra, $\a \leq \even{\g}$ a subalgebra, and $M$ a finite-dimensional $\g$-module. The theory of relative cohomology \cite{MR0080654} can be used to define relative cohomology groups $\H^n(\g,\a;M)$, which may be viewed as relative derived functors of $\Hom_{(\g,\a)}(\CC,-)$. In \cite[Theorem 2.5.2]{BKN-1} it was shown that when $\g$ is a classical Lie superalgebra and $\a = \even{\g}$, the cohomology ring $\H^\bullet(\g,\even{\g};\CC)$ is the subring $S(\odd{\g}^*)^{\even{G}}$ of invariants under a reductive group action, and is thus finitely generated over $\CC$. This paper's main result extends this work to arbitrary subalgebras $\a \leq \even{\g}$.

\begin{maintheorem}
  Let $\g = \even{\g} \oplus \odd{\g}$ be a classical Lie superalgebra, and $\a \leq \even{\g}$ an (even) subalgebra, and $M$ a $\g$-module.
  \begin{enumerate}[\indent\rm (a)]
    \item There is a spectral sequence $\{E_r^{p,q}\}$ which computes cohomology and satisfies
  \[
    E_2^{p,q}(M) \cong \H^p(\g,\even{\g};M) \otimes \H^q(\even{\g},\a;\CC) \Rightarrow \H^{p+q}(\g,\a;M)
  \]
  For $1 \leq r \leq \infty$, $E_r^{\bullet,\bullet}(M)$ is a module for $E_2^{\bullet,\bullet}(\CC)$. When $M$ is finite-dimensional, $E_2^{\bullet,\bullet}(M)$ is a Noetherian $E_2^{\bullet,\bullet}(\CC)$-module.
    \item Moreover, the cohomology ring $\H^\bullet(\g,\a;\CC)$ is a finitely-generated $\CC$-algebra.
  \end{enumerate}
\end{maintheorem}

The paper is outlined as follows. In Section \ref{sec:prelims}, Lie superalgebras, modules for Lie superalgebras, and cohomology of Lie superalgebras are defined. The pace is brisk and the interested reader will find a more thorough overview in \cite{BKN-1,Kac}. In Section \ref{sec:spectral-sequence} the author establishes finite generation of the relative cohomology ring. To do so, a first-quadrant spectral sequence as described above is constructed, similar to that of Hochschild and Serre \cite{HS-53}, pages are identified, and a standard argument is used. Additionally, the edge homomorphism of the $E_2^{\bullet,0} \to E_\infty^{\bullet,0}$ is identified as restriction, making $\H^\bullet(\g,\a;\CC)$ an integral extension of a homomorphic image of $\H^\bullet(\g,\even{\g};\CC)$. Equipped with the spectral sequence of the previous section, we devote Section \ref{sec:structure} to investigating the structure of these relative cohomology rings. For the relative cohomology ring to be Cohen-Macaulay, it is shown to be sufficient that the spectral sequence of Section \ref{sec:spectral-sequence} collapse at the $E_2$ page. This is used to compute a broad class of examples. Finally, we are in a position to systematically study support varieties for Lie superalgebras, which we do in Section \ref{sec:support-varieties}. In this section, support varieties are defined and several basic properties are stated before addressing the more difficult questions of realizability and connectedness. Our realizability theorem demonstrates a naturality between support varieties for $(\g,\even{\g})$ and those for $(\g,\a)$.

The author would like to thank his advisor, Dr Daniel Nakano for much motivation and support, and Dr William Graham for helpful conversations involving spectral sequences and Lie algebra cohomology. This work was completed as a part of the author's dissertation at University of Georgia, and was partially supported by the Research and Training Group in Algebraic Geometry, Algebra and Number Theory grant DMS-1344994 funded by the National Science Foundation.

\section{Preliminaries}
\label{sec:prelims}

This section includes basic facts about Lie superalgebras and their representation theory.

\subsection{Lie superalgebras}
\label{sec:superalgebras}

A \textit{superspace} is a $\ZZ_2$-graded vector space $V = \even{V} \oplus \odd{V}$. Elements of $\even{V}$ are called \textit{even} while elements of $\odd{V}$ are called \textit{odd}. Elements of $\even{V} \cup \odd{V}$ are called \textit{homogeneous}. For homogeneous $x \in V_i$, the \textit{degree} is $\bar x = i \in \ZZ_2$. If $V = \even{V} \oplus \odd{V}$ and $W = \even{W} \oplus \odd{W}$ are superspace then so is $\Hom_\CC(V,W)$ via
\begin{equation}\label{eq:hom-grading}
  \Hom_\CC(V,W) = \underbrace{\Hom_\CC(\even{V},\even{W}) \oplus \Hom_\CC(\odd{V}, \odd{W})}_{\even{\Hom_\CC(V,W)}} \oplus \underbrace{\Hom_\CC(\even{V},\odd{W}) \oplus \Hom_\CC(\odd{V},\even{W})}_{\odd{\Hom_\CC(V,W)}}
\end{equation}

A \textit{Lie superalgebra} is a superspace $\g  = \even{\g} \oplus \odd{\g}$ equipped with a  bilinear bracket operation $[\cdot,\cdot]:\g\otimes\g \to \g$ which respects the grading, i.e. $[\g_i,\g_j] \subseteq \g_{i+j}$, satisfying the following analogues of anticommutativity and the Jacobi identity:
\begin{enumerate}
\item $[x,y] + (-1)^{\bar x \cdot \bar y} [y,x] = 0$,
\item $[x,[y,z]] = [[x,y],z] + (-1)^{\bar x \bar y} [y,[x,z]]$.
\end{enumerate}
Note that $\even{\g}$ is a Lie algebra and $\odd{\g}$ is a $\even{\g}$-module. When $\g$ and $\h$ are Lie superalgebras, a \textit{homomorphism} $\varphi:\g\to\h$ is an even linear map satisfying $\varphi([x,y]) = [\varphi(x),\varphi(y)]$. A \textit{Lie subsuperalgebra} of $\g$ is a Lie superalgebra equipped with an injective homomorphism of Lie superalgebras $i : \a \hookrightarrow \g$. Namely, a subalgebra $\a \leq \g$ with $\a = \a \cap \even{\g} \oplus \a \cap \odd{\g}$ is a subsuperalgebra.

A \textit{classical Lie superalgebra} is a Lie superalgebra $\g = \even{\g} \oplus \odd{\g}$ such that there exists a reductive algebraic group $\even{G}$ with $\even{\g} = \Lie(\even{G})$ with an action of $\even{G}$ on $\odd{\g}$ which differentiates to the adjoint action of $\even{\g}$. Basic examples of classical Lie superalgebras may be found in \cite{BKN-1}.

\subsection{Modules for Lie superalgebras}
\label{sec:modules}

A \textit{representation} of $\g$ is a homomorphism $\varphi: \g \to \gl(V)$, which defines a \textit{$\g$-module} structure on $V$ via $g.m = \varphi(g)(m)$. There exists an associative algebra $U(\g)$, called the \emph{universal enveloping superalgebra}, such that $U(\g)$-modules correspond to $\g$-modules as just defined.

The category of $\g$-modules is not an Abelian category. To remedy this, we restrict our attention to the Abelian subcategory $\even{\Mod(\g)}$, whose objects are the same and whose morphisms are the even homomorphisms of modules for $\g$. We complement our study of $\even{\Mod(\g)}$ using the \textit{parity change functor} $\Pi: \Mod(\g) \to \Mod(\g)$ which changes the grading on a superspace; $\even{\Pi(M)} = \odd{M}$ and $\odd{\Pi(M)} = \even{M}$. Observe that $\Pi(M)$ is naturally a $\g$-module since $\End(\Pi(M)) \cong \End(M)$ as graded vector spaces. The utility of this functor comes from the fact that $\odd{\Hom_\CC(V,W)} = \even{\Hom_\CC(V,\Pi(W))}$.


\subsection{Cohomology of Lie superalgebras}
\label{sec:cohomology}

Let $\g$ be a Lie superalgebra, $\a \leq \g$ a subsuperalgebra, and $M$ a $\g$-module. We define relative cohomology groups $\H^n(\g,\a;M)$, as in \cite{BKN-1,BKN-2}, which fit into Hochschild's framework of relative homological algebra \cite{MR0080654}.

Abstractly, these cohomology groups are so-called \textit{relative Ext groups}, $\H^n(\g,\a;M) = \Ext^n_{(\g,\a)}(\CC,M)$, and arise as relative derived functors $\Hom$ in the relative category of $(U(\g),U(\a))$-modules, as in \cite{MR1923198}. In the classical way $\Ext^n_{(\g,\a)}(M,N)$ classifies equivalence classes of $n$-fold extensions of $M$ by $N$ which are split-exact upon restriction to $\a$. The usual Yoneda splice defines a product $\Ext^n_{(\g,\a)}(M,N) \times \Ext^m_{(\g,\a)}(N',M) \to \Ext^{n+m}_{(\g,\a)}(N',N)$. One may define a cup product $\Ext_{(\g,\a)}^n(M,N) \otimes \Ext_{(\g,\a)}^m(M',N') \to \Ext^{n+m}_{(\g,\a)}(M \otimes M', N \otimes N')$. When $M = M' = N = N' = \CC$ the cup product and Yoneda splice define the same product structure. Note that $\Ext_{(\g,\a)}^\bullet(M,M)$ is a ring and a module for $\Ext_{(\g,\a)}^\bullet(\CC,\CC)$ via the cup product.

A Koszul complex to compute relative cohomology is defined as follows. The $n^\text{th}$ relative cochain group is defined to be
\[
  C^n(\g,\a;M) = \Hom_\a \left( \superext{n}(\g/\a),M \right),
\]
where $\superext{n}(\even{V} \oplus \odd{V}) = \bigoplus_{p+q=n} \ext{p}(\even{V}) \otimes S^q(\odd{V})$ is the $n^\text{th}$ \emph{superexterior power}.

The coboundary $d : C^n(\g,\a;M) \to C^{n+1}(\g,\a;M)$ is, for homogeneous $f \in C^n(\g,\a;M)$ and $\w_i \in \g/\a$, given by the formula
\begin{align*}
  df(\w_0 \wedge \ldots \wedge \w_n) = &\sum_{i = 0}^n (-1)^{\tau_i(\bar \w_0, \ldots,\bar\w_n,\bar f)} \w_i .f(\w_0 \wedge \ldots \hat \w_i \ldots \wedge \w_n) \\
 &+ \sum_{i < j} (-1)^{\sigma_{i,j}(\bar \w_0, \ldots , \bar \w_n)} f([\w_i,\w_j] \wedge \w_0 \wedge \ldots \hat \w_i \ldots \hat \w_j \ldots \wedge \w_n)
\end{align*}
where parities are defined by the following formulae which differ slightly from those of \cite{BKN-1} due to a change in indexing.
\begin{align*}
  \tau_i(\alpha_0,\ldots,\alpha_n,\beta) &= i + \alpha_i (\alpha_0 + \ldots + \alpha_{i-1} + \beta)\\
  \sigma_{i,j}(\alpha_0,\ldots,\alpha_n) &= i + j + \alpha_i \alpha_j + \alpha_i(\alpha_0 + \ldots + \alpha_{i-1}) + \alpha_j (\alpha_0 + \ldots + \alpha_{j-1})
\end{align*}
As one may expect, $d \circ d = 0$ and $d$ is a homomorphism of $\a$-modules. Thus we may define relative cohomology groups
\[
  \H^n(\g,\a;M) = \frac{\ker \left( d: C^n(\g,\a;M) \to C^{n+1}(\g,\a;M) \right)}{\im \left( d: C^{n-1}(\g,\a;M) \to C^n(\g,\a;M) \right)}.
\]
The usual (absolute) cohomology groups of $\g$ with coefficients in $M$ may be recovered by considering $\H^n(\g,0;M)$.

In the special case that $\g$ is a classical Lie superalgebra and $\a = \even{\g}$, the relative cohomology ring may be identified as $\H^\bullet(\g,\even{\g};\CC) \cong S^\bullet(\odd{\g}^*)^{\even{G}}$, \cite{BKN-1}. Being the invariants of a reductive group action, this ring is finitely generated over $\CC$.

\section{Spectral sequence}
\label{sec:spectral-sequence}

\subsection{}
\label{sec:ss-intro}

Let $\g = \even{\g} \oplus \odd{\g}$ be a classical Lie superalgebra, whose cohomology relative to $\even{\g}$ was studied in \cite{BKN-1,BKN-2}. This inspires the construction of a spectral sequence modeled upon that of Hochschild-Serre \cite[\S 2]{HS-53} to compute cohomology relative to an even subsuperalgebra $\a \leq \even{\g}$ with coefficients in a finite-dimensional $\g$-module $M$. The main theorem of this section is the following.

\begin{theorem} \label{thm:spectral-sequence}
  Let $\g = \even{\g} \oplus \odd{\g}$ be a classical Lie superalgebra, $\a \leq \even{\g}$ a subalgebra and $M$ a $\g$-module.
  \begin{enumerate}[\indent\rm (a)]
    \item There exists a spectral sequence $\{E_r^{p,q}(M)\}$ which computes cohomology and satisfies
  \[
    E_2^{p,q}(M) \cong \H^p(\g,\even{\g};M) \otimes \H^q(\even{\g},\a;\CC).
  \]
  For $1 \leq r \leq \infty$, $E_r^{\bullet,\bullet}(M)$ is a module for $E_r^{\bullet,\bullet}(\CC)$. Moreover, when $M$ is finite-dimensional, $E_2^{\bullet,\bullet}(M)$ is a Noetherian $E_2^{\bullet,\bullet}(\CC)$-module.
    \item The cohomology ring $\H^\bullet(\g,\a;M)$ is finitely generated over $S(\odd{\g}^*)^{\even{G}} = \H^\bullet(\g,\even{\g};\CC)$ by homogeneous elements. As such, $\H^\bullet(\g,\a;\CC)$ is a finitely generated $\CC$-algebra.
  \end{enumerate}
\end{theorem}

\subsection{Filtration of $C^\bullet(\g,\a;M)$}
\label{sec:filtration}

Recall the decomposition
\begin{equation} \label{eq:direct-sum-cochain}
  C^n(\g,\a;M) = \bigoplus_{i + j = n} C^i\left(\even{\g},\a; \Hom_\CC \left( \superext{j}(\g/\even{\g}) ,M\right)\right).
\end{equation}
Define a descending filtration on $C^n(\g,\a;M)$ by
\begin{equation} \label{eq:filtration}
  C^n(\g,\a;M)_{(p)} = \bigoplus_{\substack{i + j = n \\ i \leq n-p}}C^i\left(\even{\g},\a;\Hom_\CC\left(\superext{j}(\g/\even{\g});M\right)\right).
\end{equation}

\begin{proposition} \label{prop:grading-properties}
  Let $\g = \even{\g} \oplus \odd{\g}$ be a Lie superalgebra, $\a \leq \even{\g}$ an even subalgebra, $M$ a $\g$-module, and $C^n(\g,\a;M)_{(p)}$ the filtration defined in \ref{eq:filtration}.
  \begin{enumerate}[\indent\rm (a)]
  \item This grading respects the differential, i.e., $d(C^n(\g,\a;M)_{(p)}) \subseteq C^{n+1}(\g,\a;M)_{(p)}$, and thus $C^\bullet(\g,\a;M)_{(p)}$ is a subcomplex of $C^\bullet(\g,\a;M)$ for all $p$.
  \item $C^n(\g,\a;M)_{(p)}$ is an $\a$-submodule of $C^n(\g,\a;M)$, so $C^\bullet(\g,\a;M)_{(p)}$ is a subcomplex of $\a$-modules.
  \item The filtration is exhaustive, i.e., $C^\bullet(\g,\a;M)_{(0)} = C^\bullet(\g,\a;M)$ and $\bigcap_{p\geq 0} C^\bullet(\g,\a;M)_{(p)} = 0$.
  \end{enumerate}
\end{proposition}
\begin{proof}
  \begin{enumerate}[(a)]
  \item Let $f \in C^n(\g,\a;M)_{(p)}$. This means $f$ vanishes when more than $n-p$ arguments belong to $\even{\g}/\a$. We wish to show that $df \in C^{n+1}(\g,\a;M)_{(p)}$, i.e., that $df$ vanishes when more than $n-p+1$ arguments belong to $\even{\g}/\a$. Let $\alpha_0, \ldots , \alpha_{n-p+1} \in \even{\g}/\a$, while $\beta_{n-p+2},\ldots,\beta_n \in \odd{\g}$. Plugging these into the coboundary formula
    \begin{align*}
      df(\alpha_0 \wedge \ldots \wedge \beta_n)
      &= \sum_{0 \leq i \leq n-p+1} (-1)^{\tau_i(-)} \alpha_i . f(\alpha_0 \wedge \ldots \hat \alpha_i \ldots \wedge \beta_n) \\
      & + \sum_{n-p+2 \leq i \leq n} (-1)^{\tau_i(-)} \beta_i .f(\alpha_0 \wedge \ldots \hat \beta_i \ldots \wedge \beta_n) \\
      &+ \sum_{0 \leq i < j \leq n-p+1 } (-1)^{\sigma_{i,j}(-)} f([\alpha_i,\alpha_j] \wedge \alpha_0 \ldots \hat \alpha_i \ldots \hat \alpha_j \ldots \wedge \beta_n) \\
      & + \sum_{\substack{0 \leq i \leq n-p+1 \\ n-p+2 \leq j \leq n}} (-1)^{\sigma_{i,j}(-)} f([\alpha_i,\beta_j]\wedge \alpha_0 \ldots \hat \alpha_i \ldots \hat \beta_j \ldots \wedge \beta_n) \\
      &+ \sum_{n-p+2 \leq i < j \leq n} (-1)^{\sigma_{i,j}(-)} f([\beta_i,\beta_j] \wedge \alpha_0 \ldots \hat \beta_i \ldots \hat \beta_j \ldots \wedge \beta_n)
    \end{align*}
    Looking at each line of the previous equation, notice that $f$ takes in, respectively, $n-p+1$, $n-p+2$, $n-p+1$, $n-p+2$, and $n-p+3$ arguments lying in $\even{\g}/\a$. Thus each term in each summation individually vanishes. Thus we conclude $df \in C^{n+1}(\g,\a;M)_{(p)}$.
  \item Let $x \in \a$, $f \in C^n(\g,\a;M)_{(p)}$. Thus $f(\w_0 \wedge \ldots \w_{n-1})$ vanishes when $n-p+1$ of the $\w_i$ belong to $\a$. writing out the definition of $(x.f)(\w_0\wedge\ldots\wedge\w_{n-1})$ we realize that each term vanishes when $n-p+1$ of the $\w_i$ belong to $\a$, and thus $x.f \in C^n(\g,\a;M)_{(p)}$.
  \item This follows from writing out the definitions and noting $C^n(\g,\a;M)_{(n+1)} = 0$
  \end{enumerate}
\end{proof}
Proposition \ref{prop:grading-properties} guarantees the existence of a spectral sequence of $\a$-modules
\[
  E_r^{p,q}(M) \Rightarrow \H^{p+q}(\g,\a;M).
\]

\subsection{Identifying pages}
\label{sec:pages}

We proceed to identify the pages $E_0$, $E_1$, and $E_2$ of the spectral sequence. The identifications of $E_0$ and
$E_1$ remain valid for any Lie superalgebra and any $\a \leq \even{\g}$. Only in identifying the $E_2$ page is the property that $\g$ is classical used.

Before identifying pages of the spectral sequence we will require a lemma whose proof may be found in \cite[Theorem 10]{HS-53}, but for convenience is reproduced below.
\begin{lemma}
  Let $\even{\g}$ be a reductive Lie algebra, $M$ be a finite-dimensional semisimple $\even{\g}$-module such that $M^{\even{\g}} = 0$. Then $\H^n(\even{\g},\a;M) = 0$ for all $n \geq 0$.
\end{lemma}
\begin{proof}
  Suppose $M$ is simple and that $n \geq 0$. The group $Z^n(\even{g},\a;M) \subseteq C^n(\even{\g},\a;M) \subseteq C^n(\even{\g};M)$ is semisimple. The group $d(C^{n-1}(\even{\g},\a;M))$ is a submodule of $Z^n(\even{\g},\a;M)$, and as such there exists a $\even{\g}$-module complement $V$ so that $Z^n(\even{\g},\a;M) = d(C^{n-1}(\even{\g},\a;M)) \oplus V$. We notice that $\even{\g} . Z^n(\even{\g},\a;M) \subseteq d(C^{n-1}(\even{\g},\a;M))$, meaning $\even{\g} . V = 0$. Thus it suffices to show that every cocycle which is annihilated by $\even{\g}$ is a coboundary.

  Since $\even{\g}$ is reductive we may write $\even{\g} = [\even{\g},\even{\g}] \oplus \z$ where $\z$ denotes the center of $\even{\g}$. Since $M$ is simple either $\z . M = 0$ or no non-zero element of $M$ is annihilated by $Z$. Let $f$ be a cocycle which is annihilated by $\even{\g}$, let $z \in \z$, and let $\w_1,\ldots,\w_n \in \even{\g}$. Then $0 = (z.f)(\w_1 \wedge \ldots \wedge \w_n) = z. f(\w_1 \wedge \ldots \wedge \w_n)$. Thus if $\z.M \neq 0$, it follows that $f = 0$. Now we may suppose $\z.M = 0$ and $M \neq 0$.

  Let $C$ be the annihilator in $\even{\g}$ of $M$, so that $C \supseteq Z$. Since the invariant submodule $M^\even{\g} = 0$, it must be the case $C \neq \even{\g}$. Now $C \cap [\even{\g},\even{\g}]$ is an ideal in the semisimple Lie algebra $[\even{\g},\even{\g}]$, meaning there must be a complementary ideal $S$. Of course, $S$ is a non-zero semisimple ideal of $\even{\g}$, which may be decomposed as $\even{\g} = S \oplus C$. Now $M\res_s$ is simple and the representation of $S$ is one-to-one. Thus the Casimir operator of this representation, $\Gamma$, is an automorphism of $M$ which commutes with all $\even{\g}$-operators on $M$. Furthermore, since $[S,C] = 0$, it is seen that for any relative cocycle $f$, $\Gamma \circ f = dg$ is a coboundary. Hence $f = \Gamma^{-1} \circ dg = d( \Gamma^{-1} \circ g)$ as desired.
\end{proof}

\begin{proposition}
  The first three pages of the spectral sequence associated to the filtration of Section \ref{sec:filtration} may be identified as follows.
  \begin{enumerate}[\indent\rm (a)]
    \itemsep.5em
  \item $E_0^{p,q} \cong C^q\left(\even{\g},\a;\Hom_\CC \left(\superext{p}(\g/\even{\g}),M\right)\right)$,
  \item $E_1^{p,q} \cong \H^q\left(\even{\g},\a;\Hom_\CC\left(\superext{p}(\g/\even{\g}),M\right)\right)$,
  \item $E_2^{p,q} \cong \H^p(\g,\even{\g};M) \otimes \H^q(\even{\g},\a;\CC)$.
  \end{enumerate}
\end{proposition}

\begin{proof}
  We proceed in steps, identifying the pages in sequence.
  \begin{enumerate}[(a)]
\item   By definition, $E_0^{p,q} = C^{p+q}(\g,\a;M)_{(p)}/C^{p+q}(\g,\a;M)_{(p+1)}$. Using the direct sum decomposition of Equation \ref{eq:filtration}, this is exactly $C^q\left(\even{\g},\a;\Hom_\CC\left(\superext{p}(\g/\even{\g}),M\right)\right)$.
\item Functoriality of the isomorphism of (a), i.e., $E_0^{p,\bullet} \cong C^\bullet(\even{\g},\a;\Hom_\CC(S^p(\g/\even{\g}),M))$ as complexes will imply their cohomologies are equal, i.e.,  $E_1^{p,q} \cong \H^q(\even{\g},\a;\Hom_\CC(S^p(\g/\even{\g}),M))$.

  To deduce functoriality of the isomorphism it will suffice to chase the following diagram.
    \[
    \begin{tikzcd}
      C^{p+q}(\g,\a;M)_{(p)} \arrow[r,"d_{(\g,\a)}"] \arrow[d,twoheadrightarrow] & C^{p+q+1}(\g,\a;M)_{(p)} \arrow[d,twoheadrightarrow] \arrow[dd,bend left=35,"\pi"] \\
      E_0^{p,q} \arrow[r,"d_0"] \arrow[d,"\cong"] & E_0^{p,q+1} \arrow[d,"\cong"] \\
      C^q(\even{\g},\a;\Hom_\CC(S^p(\odd{\g}),M)) \arrow[r,"d_{(\even{\g},\a)}"] \arrow[uu,bend left=35,"i"] & C^{q+1}(\even{\g},\a;\Hom_\CC(S^p(\odd{\g}),M))
    \end{tikzcd}
  \]
  With section $i$ corresponding to the direct sum decomposition given in Equation \ref{eq:direct-sum-cochain}. The goal is to show the composition $ \pi \circ d_{(\g,\a)}\circ i = d_{(\even{\g},\a)}$. Since $d_0$ is defined by $d_{(\g,\a)}$, this will show the bottom square commutes, resulting in an isomorphism of complexes.

  Choose $f \in C^q(\even{\g},\a;\Hom_\CC(S^p(\odd{\g}),M))$, and notice that $df$ is given by usual Lie algebra differential
  \[
df(\w_0 \wedge \ldots \w_q) = \sum_{i = 0}^q (-1)^i \w_i.f(\w_0 \wedge \ldots \hat \w_i \ldots \wedge \w_q) + \sum_{i < j} (-1)^{i+j}f([\w_i,\w_j] \wedge \w_0 \wedge \ldots \hat \w_i \ldots \hat \w_j \ldots \wedge \w_q)
  \]
  Set $\tilde f = i(f) \in C^{p+q}(\g,\a;M)$. The differential is given by the Lie superalgebra cohomology differential, and we arrive at a formula for $d_{(\g,\a)}f(\w_0 \wedge \ldots \wedge \w_{p+q})$. However, because we are taking a quotient $\pi$, it only matters how $d_{(\g,\a)}\tilde{f}$ behaves with $q+1$ even arguments and $p$ odd arguments. Thus we investigate
  \begin{align*}
    d_{(\g,\a)}f(\alpha_0\wedge \ldots \wedge \alpha_q \wedge \beta_1 \wedge \ldots \wedge \beta_p) &= \sum_{i = 0}^q (-1)^{\tau_i(-)} \alpha_i.\tilde{f}(\alpha_0\wedge \ldots \hat \alpha_0 \ldots \wedge \alpha_q \wedge \beta_1 \wedge \ldots \wedge \beta_p) \\
                                                                                                    &+ \sum_{i = q+1}^{p+q} (-1)^{\tau_i(-)} \beta_{i - q}.f(\alpha_0 \wedge \ldots \wedge \alpha_q \wedge \beta_1 \wedge \ldots \hat \beta_{i-q} \ldots \wedge \beta_p) \\
                                                                                                    &+ \sum_{0 \leq i < j \leq q} (-1)^{\sigma_{i,j}(-)}f([\alpha_i,\alpha_j] \wedge \alpha_0 \ldots \hat \alpha_i \ldots \hat \alpha_j \ldots \beta_p) \\
                                                                                                    &+ \sum_{\substack{0 \leq i \leq q \\ q+1 \leq j \leq p+q}} (-1)^{\sigma_{i,j}(-)} f([\alpha_i,\beta_{j-q}] \wedge \alpha_0 \ldots \hat \alpha_i \ldots \hat \beta_{j-q} \ldots \beta_p) \\
    &+ \sum_{q+1 \leq i < j \leq p+q} (-1)^{\sigma_{i,j}(-)} f([\beta_{i-q},\beta_{j-q}]\wedge \alpha_0 \ldots \hat \beta_i \ldots \hat \beta_j \ldots \wedge \beta_p)
  \end{align*}
  By construction, $\tilde{f}$ vanishes unless exactly $q$ arguments are even and $p$ arguments are odd. This only occurs in the first, third, and fourth lines of the preceding sum. Working out the relevant signs yields
\[
  \tau_i(\underbrace{\bar 0,\ldots,\bar 0}_{q+1},\underbrace{\bar 1,\ldots,\bar 1}_p,\bar f) = i \text{ when } i \leq q
\]
\[
    \sigma_{i,j}(\underbrace{\bar 0,\ldots,\bar 0}_{q+1},\underbrace{\bar 1,\ldots,\bar 1}_p) = \begin{cases}
      i + j &\text{ if } i,j \leq q \\
      i - q - 1 &\text{ if } i \leq q, j \geq q+1
    \end{cases}
 \]
So the previous equation for $d_{(\g,\a)}\tilde f$ becomes
  \begin{align*}
    d_{(\g,\a)}f(\alpha_0\wedge \ldots \wedge \alpha_q \wedge \beta_1 \wedge \ldots \wedge \beta_p) &= \sum_{i = 0}^q (-1)^{i} \alpha_i.\tilde{f}(\alpha_0\wedge \ldots \hat \alpha_0 \ldots \wedge \alpha_q \wedge \beta_1 \wedge \ldots \wedge \beta_p) \\
                                                                                                    &+ \sum_{0 \leq i < j \leq q} (-1)^{i+j}f([\alpha_i,\alpha_j] \wedge \alpha_0 \ldots \hat \alpha_i \ldots \hat \alpha_j \ldots \beta_p) \\
                                                                                                    &- \sum_{\substack{0 \leq i \leq q \\ q+1 \leq j \leq p+q}} (-1)^{i} f(\alpha_0 \ldots \hat \alpha_i \ldots \wedge \alpha_q \wedge [\alpha_i,\beta_{j-q}] \wedge \beta_1\ldots \hat \beta_{j-q} \ldots \beta_p) \\
  \end{align*}
  Now if we compute $d_{(\even{\g},\a)}f$, accounting for the action on $\Hom_\CC(S^p(\odd{\g}),M)$, we arrive at the same formula.
\item Notice first that by semisimplicity $\Hom_\CC(S^n(\g/\even{\g}),M) \cong \Hom_{\even{\g}}(S^n(\g/\even{\g}),M) \oplus V$ where $V$ is some complement with $V^{\even{\g}} = 0$. By the lemma,
  \[
    E_1^{p,q} \cong \H^q(\even{\g},\a;\Hom_{\even{\g}}(S^p(\g/\even{\g}),M)) \oplus \H^q(\even{\g},\a;V) = \H^q(\even{\g},\a;\Hom_{\even{\g}}(S^p(\g/\even{\g}),M)).
  \]
  Because $\even{\g}$ acts trivially on $\Hom_{\even{\g}}(S^p(\g/\even{\g}),M)$, we may conclude that $E_1^{p,q} \cong \H^q(\even{\g},\a;\CC) \otimes \Hom_{\even{\g}}(S^p(\g/\even{\g}),M)$. This association is functorial, i.e., induces an isomorphism $E_1^{\bullet,q} \cong \H^q(\even{\g},\a;\CC) \otimes \Hom_\CC(S^\bullet(\g/\even{\g}),M)$ as complexes. Therefore, we may conclude that $E_2^{p,q} \cong \H^q(\even{\g},\a;\CC) \otimes \H^p(\g,\even{\g};M)$.
  \end{enumerate}
\end{proof}

\subsection{Proof of main theorem}
\label{sec:proof-of-fg}

Here we present a proof of the final statement of Theorem \ref{thm:spectral-sequence}. The line of reasoning follows the one given in \cite[Theorem 1.4]{FP-infinitesimal}.

\begin{proposition}
  Under the hypotheses of Theorem \ref{thm:spectral-sequence}, when $M$ is a finite-dimensional $\g$-module, $E_r^{\bullet,\bullet}(M)$ is a Noetherian $E_r^{\bullet,\bullet}(\CC)$-module, for $2 \leq r \leq \infty$.
\end{proposition}

\begin{proof}
  Let $M$ be a finite-dimensional $\g$-module. $E_2^{\bullet,\bullet}(M)$ is a Noetherian $S^\bullet(\odd{\g}^*)^{\even{G}}$-module via the map
  \[
    S^\bullet(\odd{\g}^*)^{\even{G}} \hookrightarrow E_2^{\bullet,0}(\CC) \subseteq E_2^{\bullet,\bullet}(\CC).
  \]
  $E_\infty^{\bullet,\bullet}(M)$, being a section of $E_2^{\bullet,\bullet}(M)$ is a Noetherian $S^\bullet(\odd{\g}^*)^{\even{G}}$-module via the map
  \[
    S^\bullet(\odd{\g}^*)^{\even{\g}} \to E_\infty^{\bullet,0}(\CC) \subseteq E_\infty^{\bullet,\bullet}(\CC).
  \]
  Consequently, $E_\infty^{\bullet,\bullet}(M)$ is a Noetherian $E_\infty^{\bullet,\bullet}(\CC)$-module.

\end{proof}

\begin{proposition}
  The edge homomorphism of the spectral sequence corresponds to restriction.
\end{proposition}
\begin{proof}
The restriction map $C^n(\g,\even{\g};\CC) \xrightarrow{\resmap} C^n(\g,\a;\CC)$ induces a map on cohomology $\H^n(\g,\even{\g};\CC) \xrightarrow{\resmap^*} \H^n(\g,\a;\CC)$. Because $\resmap$ repects the filtration of Section \ref{sec:filtration}, the map $\resmap^*$ will respect the induced filtration on cohomology, i.e., $F^p \H^n(\g,\even{\g};\CC) \xrightarrow{\resmap^*} F^p\H^n(\g,\a;\CC)$. This descends to a map on the associated graded of each cohomology ring, which may be precomposed with the projection onto associated graded as follows
  \[
    \H^n(\g,\even{\g};\CC) \to \Gr\left(\H^n(\g,\even{\g};\CC) \right) \to \Gr\left( \H^n(\g,\a;\CC) \right)
  \]

\end{proof}

\subsection{}
\label{sec:computation}

  In many instances, Lie superalgebra cohomology $\H^\bullet(\g;\CC) = \H^\bullet(\g,0;\CC)$ will vanish in all but finitely many degrees (see \cite{fuks1984cohomology} or \cite[Th\'eor\`eme 5.3]{MR1450424}), leading one to conclude the ring has Krull dimension zero and thus uninteresting geometry. Here it is shown that for $\g = \gl(1|1)$ and $\a$ generated by $\operatorname{diag}(1 \mid 1) \in \gl(1|1)$, $\H^\bullet(\g,\a;\CC)$ is nonzero in infinitely many degrees. From this, we may conclude $\H^\bullet(\g,\a;\CC)$ has positive Krull dimension. This is an especially nice case; $\a$ acts trivially on $\gl(1|1)$ so every map $\superext{n}(\g/\a) \to \CC$ is $\a$-invariant.

  Take the basis for $\gl(1|1)/\a$
  \[
    \alpha = \begin{pmatrix} 1 & 0 \\ 0 & 0\end{pmatrix}\text{ , } \beta_1 = \begin{pmatrix}0 & 1 \\ 0 & 0\end{pmatrix} \text{ , } \beta_2 = \begin{pmatrix}0 & 0 \\ 1 & 0\end{pmatrix}
  \]
  $\superext{2n}(\g/\a)$ has basis $\{\alpha \otimes \beta_1^i \beta_2^j\}_{i + j + 1 = n} \cup \{\beta_1^i \beta_2^j\}_{i+j = n}$. Consider $f \in C^{2i}(\g,\a;\CC)$ which maps $\beta_1^n \beta_2^n$ to 1 and all other basis vectors to zero. Since $\CC$ has the trivial action, $df$ has the form
  \[
    df(\w_0 \wedge \ldots \wedge \w_{2n}) = \sum_{i = 0}^p (-1)^{\sigma_{i,j}(\bar \w_0,\ldots,\bar\w_{2n})} f([\w_i,\w_j] \wedge \w_0 \wedge \ldots \hat \w_i \ldots \hat \w_j \ldots \wedge \w_{2n})
  \]
  By inspection, $df$ will vanish on all basis vectors $\beta_1^i \beta_2^j$ and $df(\alpha \otimes \beta_1^i \beta_2^j) = (i-j) f(\beta_1^i \beta_2^j)$. This is $0$ when $i,j \neq n$ by definition of $f$, and when $i = j = n$ this is zero because the coefficient vanishes. So $f$ is a cocycle.

  Suppose $dg = f$ for some $g \in C^{2n-1}(\g,\a;\CC)$. Then we compute $dg(\beta_1^n \beta_2^n)$, which is a sum of terms of the form $(-1)^{\sigma_{i,j}(-)} g([\beta_k, \beta_l] \wedge \beta_1^{n_1} \wedge \beta_2^{n_2}$, each of which vanishes individually so that $dg(\beta_1^n \beta_2^n) = 0$.

  Therefore, $f$ is \emph{not} a coboundary. So for every $n \geq 2$, $\H^{2n}(\g,\a;\CC) \neq 0$. This shows that cohomology relative to an even subalgebra heuristically lies somewhere between the results of Fuks-Leites \cite{fuks1984cohomology} and Boe-Kujawa-Nakano \cite{BKN-1}.

\section{Structure of cohomology rings}
\label{sec:structure}

\subsection{Collapsing at $E_2$}
\label{sec:collapse}

The following theorem is motivated by \cite[Proposition 3.1]{MR3233523}. The reader should recall that an algebra $A$ is \emph{Cohen-Macaulay} if there is a polynomial subalgebra over which $A$ is a finite and free module, see \cite[\S 5.4]{MR1634407}.

\begin{proposition}
  \label{prop:collapse-CM}
  Let $\g = \even{\g} \oplus \odd{\g}$ be a classical Lie superalgebra, and $\a \leq \even{\g}$ a subalgebra. If the spectral sequence constructed in Section \ref{sec:spectral-sequence} collapses at $E_2$ (i.e., if $E_2^{\bullet,\bullet}(\CC) \cong E_\infty^{\bullet,\bullet}(\CC)$), then $\H^\bullet(\g,\a;\CC)$ is a Cohen-Macaulay ring.
\end{proposition}
\begin{proof}
  The spectral sequence $E_2^{\bullet,\bullet} = E_\infty^{\bullet,\bullet}$ is a filtered version of the cohomology ring $\H^\bullet(\g,\a;\CC)$. As such, if $\zeta \in E_2^{i,j}$ and $\eta \in E_2^{r,s}$, then $\zeta \cdot \eta \in \sum_{\ell \geq 0} E_2^{i + r + \ell, j + s - \ell}$. Because of this, for any $m \geq 0$, the direct sum of the lowest $m$ rows, denoted $U_m = \sum_{q \leq m} E_2^{\bullet,q}$, is a module for the bottom row $U_0 = E_0^{\bullet,0} \cong \H^0(\even{\g},\a;\CC) \otimes S^\bullet(\odd{\g}^*)^{\even{G}} \cong S^\bullet(\odd{\g}^*)^{\even{G}}$, which by \cite{MR0347810} is a Cohen-Macaulay ring. Because the spectral sequence collapses, $E_2 = E_\infty$ and the quotients $U_m / U_{m-1} \cong \H^m(\even{\g},\a;\CC) \otimes S^\bullet(\odd{\g}^*)^{\even{G}}$ are free $S^\bullet(\odd{\g}^*)^{\even{G}}$-modules. This means the quotient maps $U_m \to U_m / U_{m-1}$ split as maps of $S(\odd{\g}^*)^{\even{G}}$-modules and the proposition follows.
\end{proof}

\subsection{Applications}
\label{sec:examples}

In this section we present some applications in which we use the spectral sequence of Section \ref{sec:spectral-sequence} to compute Krull dimensions of cohomology rings in particularly nice cases. The reader should notice these results rely on deep results from representation theory in the relative Category $\O$ (cf. \cite[\S 8]{MR2428237}).

\begin{theorem} \label{prop:E2-collapse}
  Let $\g = \even{\g} \oplus \odd{\g}$ be a classical Lie superalgebra such that $S^\bullet(\odd{\g})^{\even{G}}$ vanishes in odd degrees, and $\l \leq \even{\g}$ a standard Levi subalgebra (i.e., nonzero and generated by simple roots). The following hold.
  \begin{enumerate}[\indent\rm (a)]
    \item The spectral sequence of Section \ref{sec:spectral-sequence} collapses at the $E_2$ page and $E_2^{\bullet,\bullet}(\CC) \cong E_\infty^{\bullet,\bullet}(\CC)$.
    \item $\H^\bullet(\g,\l;\CC)$ is Cohen-Macualay,
    \item $\krdim \H^\bullet(\g,\even{\g};\CC) = \krdim \H^\bullet(\g,\l;\CC)$.
  \end{enumerate}
\end{theorem}

\begin{proof}
  We establish (a). Parts (b) and (c) follow by application of Proposition \ref{prop:collapse-CM}.

  Let $\g = \even{\g} \oplus \odd{\g}$ be a classical Lie superalgebra such that $S^\bullet(\odd{\g}^*)^{\even{G}}$ is zero in odd degrees, and $\l \leq \even{\g}$ a Levi subalgebra. According to the Kazhdan-Lusztig conjectures\footnote{When $\h \leq \even{\g}$ is a Cartan subalgebra, $\Ext^n_\O(M,N) \cong \Ext^n_{(\even{\g},\h)}(M,N)$ (see \cite[Theorem 6.15]{MR2428237}). The fact that $\Ext^n_{(\even{\g},\h)}(\CC,\CC)$ vanishes in odd degrees follows from \cite{MR1245719}.}, $\H^\bullet(\even{\g},\l;\CC)$ is only nonzero in even degrees. Section \ref{sec:pages} realizes the $E_2$ page of the Hochschild-Serre spectral sequence as
\[
  E_2^{p,q}(\CC) \cong \H^q(\even{\g},\l;\CC) \otimes S^p(\odd{\g}^*)^{\even{G}}.
\]
Because the differential $d_2:E_2^{p,q} \to E_2^{p+2,q-1}$ descends one row, either $E_2^{p,q} = 0$ or $E_2^{p+2,q-1} = 0$. In either case, $d_2 = 0$ and thus $E_3^{p,q} = E_2^{p,q}$ meaning that $E_3^{p,q}$ vanishes unless $p$ and $q$ are both even. By a similar argument, the differential $d_3:E_3^{p,q} \to E_3^{p+3,q-2}$ must be zero since one of $E_3^{p,q}$ or $E_3^{p+3,q-2}$ will have odd horizontal coordinate and thus be zero. So $E_3^{p,q} \cong E_4^{p,q}$. By induction, this trend continues to arrive at the conclusion that $E_2^{p,q} \cong E_\infty^{p,q}$. This yields the following statement.
\end{proof}

This example restricts to the case that cohomology of $\g$ relative to $\even{\g}$ vanishes in odd degrees. While this may seem restrictive, \cite[Table 1]{BKN-1} reveals that there are a great many classical Lie superalgebras whose cohomology lives in even degree.

\begin{corollary}
  Let $\g = \even{\g} \oplus \odd{\g}$ be a Lie superalgebra of type $\gl(m|n)$, $\sl(m|n)$, $\mathfrak{psl}(2n|2n)$, $\mathfrak{osp}(2m+1|2n)$, $\mathfrak{osp}(2m|2n)$, $P(4\ell - 1)$, $D(2,1;\alpha)$, $G(3)$, or $F(4)$. Let $\l \leq \even{\g}$ be a standard Levi subalgebra. The following hold:
  \begin{enumerate}[\indent\rm(a)]
  \item $\H^\bullet(\g,\l;\CC)$ is a Cohen-Macaulay ring.
  \item $\krdim \H^\bullet(\g,\l;\CC) = \krdim S^\bullet(\odd{\g}^*)$.
  \end{enumerate}
\end{corollary}

\section{Support Varieties}
\label{sec:support-varieties}

\subsection{}Let $\g$ be a classical Lie superalgebra with $\a \leq \even{\g}$ a subalgebra. We showed in Theorem \ref{thm:spectral-sequence} that $\H^\bullet(\g,\a;\CC)$ is a finitely-generated graded-commutative $\CC$-algebra. Therefore, the subring of even cohomology classes $\H^{ev}(\g,\a;\CC) = \bigoplus \H^{2\bullet}(\g,\a;\CC)$ is commutative and finitely-generated over $\CC$. The \textit{cohomology variety} of $\g$ relative to $\a$ is the algebraic variety
\[
  \V_{(\g,\a)}(\CC) = \maxspec\left(\H^{ev}(\g,\a;\CC)\right).
\]
Note that since $\H^{ev}(\g,\a;\CC)$ is graded we just as well could have looked at the projectivization of $\V_{(\g,\a)}(\CC)$. When dealing with questions of connectivity it will be advantageous to use the projectivization, but in other contexts we will focus exclusively on the conical affine variety.

Recall that $\Ext_{(\g,\a)}^\bullet(M,M)$ is a module over $\H^\bullet(\g,\a;\CC)$, so its annihilator defines a subvariety called the \textit{support variety} of $M$, denoted
\[
  \V_{(\g,\a)}(M) = \Z\left({\Ann_{\H^{ev}(\g,\a;\CC)} \Ext_{(\g,\a)}^\bullet(M,M)}\right) \subseteq \V_{(\g,\a)}(\CC)
\]
where $\Z(I)$ denotes the vanishing set of $I$.

An alternative definition of the support variety is
\[
  \V_{(\g,\a)}(M) = \left\{~ \m \in \V_{(\g,\a)}(\CC) \mid \Ext_{(\g,\a)}^\bullet(M,M)_{\m} \neq 0 ~\right\}.
\]
The following are basic properties whose proof is standard and may be found in \cite{FP-unipotent}.

\begin{enumerate}
  \item For any $\g$-module $M$, $\V_{(\g,\a)}(M)$ is a closed, conical subvariety of $\V_{(\g,\a)}(\CC)$.
  \item For any $\g$-modules $M_1$ and $M_2$, $\V_{(\g,\a)}(M_1 \oplus M_2) = \V_{(\g,\a)}(M_1) \cup \V_{(\g,\a)}(M_2)$.
  \item Whenever $0 \to M_1 \to M_2 \to M_3 \to 0$ is a short exact sequence of $\g$-modules, and $\sigma \in \mathfrak{S}_3$ is a permutation of three letters, $\V_{(\g,\a)}(M_{\sigma(1)}) \subseteq \V_{(\g,\a)}(M_{\sigma(2)}) \cup \V_{(\g,\a)}(M_{\sigma(3)})$.
\end{enumerate}

In this section, we use the realization map $\Phi:\V_{(\g,\a)}(\CC) \to \V_{(\g,\even{\g})}(\CC)$ induced by restriction $\resmap: \H^\bullet(\g,\even{\g};\CC) \to \H^\bullet(\g,\a;\CC)$ to determine properties of (the image of) $\V_{(\g,\a)}(M)$. This has the advantage of taking the elusive, abstract support variety and embedding it inside of something concrete -- indeed, $\V_{(\g,\even{\g})}(\CC)$ is simply the set of closed orbits of the action $\even{G}$ on $\odd{\g}$.

\subsection{Realizability}
\label{sec:realizability}

In this section we address the question of realizability, initially studied by Carlson \cite{MR723070}. As we are using results of Bagci-Kujawa-Nakano \cite{MR2448087}, we need additional assumptions on the Lie superalgebra $\g$, namely we require the superalgebra is \emph{stable} and \emph{polar} in addition to being classical. These assumptions originate in geometric invariant theory, and hold for $\gl(m|n)$ -- see \cite[\S 3.2-3.3]{BKN-1} for a thorough description.

\begin{definition}
  Let $\g = \even{\g} \oplus \odd{\g}$ be a classical, stable, and polar Lie superalgebra with $\a \leq \even{\g}$ a subalgebra. We say $(\g,\even{\g})$-module $M$ is \emph{natural} (with respect to $\a$) if $\V_{(\g,\even{\g})}(M) \cap \Phi\left(\V_{(\g,\a)}(\CC)\right) = \Phi\left(\V_{(\g,\a)}(M)\right)$. The subalgebra $\a$ is \emph{natural} if every $\g$-module is natural with respect to $\a$.
\end{definition}

The paper of Bagci-Kujawa-Nakano \cite[Theorem 8.8.1]{MR2448087} demonstrated that every closed conical subvariety of $\V_{(\g,\even{\g})}(\CC)$ is realized as the support variety of a $(\g,\even{\g})$-module.

\begin{proposition}
  Let $\g = \even{\g} \oplus \odd{\g}$ be a classical, stable, and polar Lie superalgebra with $\a \leq \even{\g}$ a natural subalgebra. Let $X \subseteq \V_{(\g,\a)}(\CC)$ be a closed, conical subvariety. There exists a $(\g,\a)$-module $M$ such that $\Phi\left(\V_{(\g,\a)}(M)\right) = \Phi(X)$.
\end{proposition}
\begin{proof}
  The realization theorem holds for $(\g,\even{\g})$-modules, so choose $M$ such that $\V_{(\g,\even{\g})}(M) = \Phi(X)$. By naturality, $\Phi(\V_{(\g,\a)}(M)) = \Phi(\V_{(\g,\a)}(\CC)) \cap \V_{(\g,\even{\g})}(M) = \Phi(X)$.
\end{proof}

\subsection{Tensor products}
\label{sec:tensor-products}

A tensor product theorem gives us the ability to geometrically control the support theory of tensor products of modules. Historically, this has been a very elusive property of support varieties, often times requiring support varieties recognized in some other way. For example, in the case of finite groups, the tensor product theorem was not shown until support varieties were determined to be isomorphic to the very concrete rank varieties \cite{MR621284}.

In this section, we circumvent this issue by considering only superalgebras which satisfy the tensor product theorem relative to $(\g,\even{\g})$, and using the realization map to intersect supports of $(\g,\a)$-modules inside $\V_{(\g,\even{\g})}(\CC)$.

\begin{definition}
  Let $\g = \even{\g} \oplus \odd{\g}$ be a Lie superalgebra with subalgebra $\a \leq \even{\g}$. The pair $(\g,\a)$ is said to satisfy the \emph{tensor product theorem} if $\V_{(\g,\a)}(M \otimes N) = \V_{(\g,\a)}(M) \cap \V_{(\g,\a)}(N)$ for all modules $M,N$.
\end{definition}

Lehrer-Nakano-Zhang proved the tensor product theorem hold for the pair $(\gl(m|n),\even{\gl(m|n)})$, \cite[Theorem 5.2.1]{MR2836115}

\begin{proposition}
  Let $\g = \even{\g} \oplus \odd{\g}$ be a Lie superalgebra which satisfies the tensor product theorem relative to $\even{\g}$, and $\a \leq \even{\g}$ a natural subalgebra of $\g$.  Then $\Phi(\V_{(\g,\a)}(M \otimes N)) = \Phi(\V_{(\g,\a)}(M)) \cap \Phi(\V_{(\g,\a)}(N))$.
\end{proposition}
\begin{proof}
  One has:
  \begin{align*}
    \Phi(\V_{(\g,\a)}(M \otimes N)) &= \Phi(\V_{(\g,\a)}(\CC)) \cap \V_{(\g,\even{\g})}(M \otimes N) \\
                                   &= \left(\Phi(\V_{(\g,\a)}(\CC)) \cap \V_{(\g,\even{\g})}(M)\right) \cap \left( \Phi(\V_{(\g,\a)}(\CC)) \cap \V_{(\g,\even{\g})}(N)\right) \\
    &= \Phi(\V_{(\g,\a)}(M)) \cap \Phi(\V_{(\g,\a)}(N)).
  \end{align*}
\end{proof}

\subsection{Connectedness of support varieties}
\label{sec:connectivity}

This subsection investigates connectedness of support varieties, motivated by Benson's presentation \cite{MR1634407}.

\begin{proposition}
  Let $\g = \even{\g} \oplus \odd{\g}$ be a classical, stable, and polar Lie superalgebra with $\a \leq \even{\g}$ a natural subalgebra. Suppose $\Phi(\V_{(\g,\a)}(M)) = X \cup Y$ with $X \cap Y = \{0\}$. Then there exist modules $M_1$ and $M_2$ such that $M = M_1 \oplus M_2$, $X = \Phi(\V_{(\g,\a)}(M_1))$, $Y = \Phi(\V_{(\g,\a)}(M_2))$, and \[
    \Phi(\V_{(\g,\a)}(M)) = \Phi(\V_{(\g,\a)}(M_1)) \cup \Phi(\V_{(\g,\a)}(M_2)).
  \]
\end{proposition}
\begin{proof}
  By realizability for $(\g,\even{\g})$, because $\Phi(\V_{(\g,\a)}(M))$ is a closed conical subvariety of $\V_{(\g,\even{\g})}(M)$, there exist $M_1$ and $M_2$ such that $\Phi(\V_{(\g,\a)}(M)) = \V_{(\g,\even{\g})}(M_1) \cup \V_{(\g,\a)}(M_2)$. Using this fact, we may compute:
  \begin{align*}
    \Phi(\V_{(\g,\a)}(M)) &= \V_{(\g,\even{\g})}(M_1) \cup \V_{(\g,\even{\g})}(M_2) \\
                         &= \left(\Phi(\V_{(\g,\a)}(\CC)) \cap \V_{(\g,\even{\g})}(M_1)\right) \cup \left(\Phi(\V_{(\g,\a)}(\CC)) \cap \V_{(\g,\even{\g})}(M_2)\right) \\
    &= \Phi(\V_{(\g,\a)}(M_1)) \cup \Phi(\V_{(\g,\a)}(M_2)).
  \end{align*}
\end{proof}


\section{Open questions}
\label{sec:open-questions}

In this section we present problems motivated by our study of support varieties for Lie superalgebras.

\begin{enumerate}
\item For which subsuperalgebras $\a \leq \g$ is the relative cohomology ring $\H^\bullet(\g,\a;\CC)$ a finitely-generated $\CC$-algebra?
\item   The converse to Proposition \ref{prop:collapse-CM} holds. Namely, if $\g = \even{\g} \oplus \odd{\g}$ is a classical Lie superalgebra, with $\a \leq \even{\g}$. If the cohomology ring $\H^\bullet(\g,\a;\CC)$ is Cohen-Macaulay, is it true that the spectral sequence of Section \ref{sec:spectral-sequence} must collapse at the $E_2$ page?
\item Boe, Kujawa, and Nakano realized that when $\g = \gl(m|n)$ and $M$ is a simple $\g$-module, $\dim \V_{(\g,\even{\g})}(M)$ is equal to the (combinatorial) \emph{atypicality} of the module, $\atyp (M)$. Is there any combinatorial interpretation for the dimension $\dim \V_{(\g,\a)}(M)$ for more general $(\g,\a)$ with $M$ a simple $\g$-module?
\item Under what conditions is the map $\Phi: \V_{(\g,\a)}(\CC) \to \V_{(\g,\even{\g})}(\CC)$ induced by restriction a closed embedding? When is this map natural, in the sense that $\Phi(\V_{(\g,\a)}(M)) = \V_{(\g,\even{\g})}(M) \cap \Phi(\V_{(\g,\a)}(\CC))$?
Can we classify all $M$ such that $\Phi(\V_{(\g,\a)}(M)) = \V_{(\g,\even{\g})}(M) \cap \Phi(\V_{(\g,\a)}(\CC))$?
\end{enumerate}

\bibliographystyle{plain}{}
\bibliography{references}

\end{document}